\documentclass[a4paper,11pt]{article}

\usepackage[latin1]{inputenc}
\usepackage[T1]{fontenc}
\usepackage[normalem]{ulem}
\usepackage[english]{babel}
\usepackage{verbatim}
\usepackage{graphicx}
\usepackage{enumerate}
\usepackage{amsmath,amssymb,amsfonts,amsthm}
\usepackage{rotating}
\usepackage{float}
\usepackage{array}
\usepackage{multicol}

\newtheorem{theorem}{Theorem}[section]
\newtheorem*{theoremA}{Theorem A}
\newtheorem*{theoremB}{Theorem B}
\newtheorem{lemma}[theorem]{Lemma}

\theoremstyle{definition}
\newtheorem{defn}[theorem]{Definition}
\newtheorem{rem}[theorem]{Remark}
\newtheorem{example}[theorem]{Example}

\input xy
\xyoption{all}

\DeclareMathOperator{\id}{id}

\DeclareMathOperator{\map}{Map}

\DeclareMathOperator{\Top}{Top}
\DeclareMathOperator{\Ob}{Ob}

\DeclareMathOperator*{\colim}{colim}

\DeclareMathOperator*{\holim}{holim}

\DeclareMathOperator{\hofib}{hof}
\DeclareMathOperator{\hoc}{cof}
\DeclareMathOperator{\conn}{Conn}

\newdir^{ (}{{}*!/-5pt/@^{(}}

\begin{document}
\begin{center}\LARGE{Homotopy limits of spaces and connectivity}
\end{center}

\begin{center}\large{Emanuele Dotto}

\end{center}
\vspace{.3cm}

\begin{quote}
\textsc{Abstract}. This paper contains two results on how homotopy limits of topological spaces interact with connectivity. The first is a formula for the connectivity of the homotopy limit of diagrams shaped over suitably finite categories, in terms of the connectivity of the spaces in the diagram. The second result shows that the homotopy fiber of the restriction map induced on homotopy limits by a functor of indexing categories is itself a homotopy limit, indexed over a Grothendieck construction.
\end{quote}

\section*{Introduction}

It is well known that the homotopy colimit of a diagram of pointed $n$-connected spaces is itself $n$-connected. However, calculations of the connectivity of homotopy limits do not seem to be present in the literature. The first result of this paper is the following formula, expressing the connectivity of the homotopy limit of a diagram of pointed spaces in terms of the connectivities of the spaces in the diagram.

\begin{theoremA}
Let $X\colon I\rightarrow \Top_\ast$ be a diagram of pointed spaces. If the nerves of the over categories $NI/_i$ are finite dimensional simplicial sets, the homotopy limit $\holim_I X$ is at least
\[\min_{i\in \Ob I}\big(\conn X_i-\dim NI/_i\big)\]
connected.
\end{theoremA}

The second result is a refinement of the classical homotopy cofinality theorem for homotopy limits, and it gives a description of the homotopy fiber of the restriction map on homotopy limits over suitable basepoints. It has been pointed out to the author by an anonymous reviewer that it is a direct consequence of the Fubini's theorem for homotopy limits of \cite{voj}. Here we give a direct and explicit proof of this theorem in the special case of homotopy cofibers.
We use the space of natural transformations
\[\holim_{I}X=\hom(BI/_{-},X)\]
of Bousfield-Kan as a model for the homotopy limit, where $BI/_{-}\colon I\rightarrow\Top$ is the diagram that associates to an object $i$ the classifying space of the over category $I/_i$. If $F\colon I\rightarrow J$ is a functor of small categories and $X\colon J\rightarrow \Top$ is a diagram of spaces, restriction along $F$ defines a map
\[F^\ast\colon \holim_JX\longrightarrow \holim_IF^\ast X\]
where $F^\ast X$ is the composite functor $X\circ F$. The classical cofinality theorem states that if the over categories $F/_j$ are contractible for every object of $J$, the restriction map $F^\ast$ is a weak equivalence (see e.g. \cite{BK},\cite{Hirsch}). Here we describe the homotopy fiber of $F^\ast$ over a constant natural transformation in $\holim_IF^\ast X$. A natural transformation $\ast\rightarrow F^\ast X$, induces an extension $\overline{X}\colon\hoc F\rightarrow \Top$ of $X$ to the homotopy cofiber of $F$. Here $\hoc F$ is a category defined as a Grothendieck construction (see \ref{grot}), and it fits into a sequence $I\stackrel{F}{\rightarrow}J\rightarrow \hoc F$
whose nerve is a model for the mapping cone of the nerve of $F$.

\begin{theoremB}[{\cite[31.5]{voj}}]
Let $F\colon I\rightarrow J$ be a functor, $X\colon J\rightarrow\Top$ a diagram of spaces and $\ast\rightarrow F^{\ast}X$ a natural transformation. There is a homotopy cartesian square of spaces
\[\xymatrix{\displaystyle\holim_{\hoc F}\overline{X}\ar[r]^{\iota^\ast}\ar[d]&\displaystyle\holim_JX\ar[d]^{F^{\ast}}\\
\ast\ar[r]&\displaystyle\holim_IF^{\ast}X
}\]
where the bottom horizontal map is the constant natural transformation $BI/_{-}\rightarrow\ast\rightarrow F^{\ast}X$. 
\end{theoremB}
By the first result, if $X_{F(i)}$ is at least $\dim I/i$-connected the homotopy limit $\holim_IF^{\ast}X$ is connected, and therefore $\holim_{\hoc F}\overline{X}$ is the unique homotopy fiber of $F^\ast$ up to equivalence.

One could try to combine these two results to calculate the connectivity of the restriction map $F^\ast$ in terms of the connectivities of the vertices of $X$. However, this would only show that $F^\ast$ is an isomorphism in homotopy groups up to a range where both the homotopy groups of $\holim_JX$ and $\holim_IF^{\ast}X$ are zero.

\section{Connectivity of homotopy limits}

We prove the following formula, expressing the connectivity of the homotopy limit of a diagram of pointed spaces in terms of the connectivities of the spaces in the diagram.

\begin{theoremA}
Let $X\colon I\rightarrow \Top_\ast$ be a diagram of pointed spaces. If the nerves $NI/_i$ are finite dimensional simplicial sets, the homotopy limit $\holim_I X$ is at least
\[\min_{i\in \Ob I}\big(\conn X_i-\dim NI/_i\big)\]
connected.
\end{theoremA}

\begin{example}
Let $\mathcal{P}(n_+)$ be the poset category of subsets of the pointed sets $\underline{n}_+=\{+,1,\dots,n\}$ ordered by inclusion, and $\mathcal{P}_0(n_+)$ the subposet of non-empty subsets. The nerve $N\mathcal{P}_0(n_+)$ is the subdivided $n$-simplex $\Delta^{n}$.
\begin{itemize}
\item The category $\mathcal{P}_0(1_+)$ is the pullback category $\bullet \rightarrow \bullet\leftarrow\bullet$. By theorem A the connectivity of a homotopy pullback $X\rightarrow Y\leftarrow Z$ is at least
\[\min\{\conn X,\conn Z,\conn Y-1\}\]
This can be also easily verified by analyzing the long exact sequence in homotopy groups induced by the homotopy fiber of $X\rightarrow Y$.
\item More generally, Theorem A shows that the homotopy limit of a diagram $X\colon \mathcal{P}_0(n_+)\rightarrow \Top_\ast$ is at least
\[\min_{\emptyset\neq U\subset \underline{n}_+}\conn X_i-|U|+1\]
connected.
\end{itemize}
\end{example}

\begin{proof}
For any pointed map $g\colon S^{k}\rightarrow \hom_\ast(K,X)$, with $k$ smaller than the range of the statement, we need to build an extension of $g$ to the $(k+1)$-disc. By the standard adjunction this is the same as solving the extension problem
\[\xymatrix{K\wedge D^{k+1}\ar@{-->}[r]^-{\widetilde{f}}& X\\
K\wedge S^{k}\ar[u]\ar[ur]_-{\widetilde{g}}
}\]
in the category of diagrams of pointed spaces $Top_{\ast}^I$. We define the extension $\widetilde{f}$ by induction on a filtration of the objects of $I$ induced by the degree function $\deg\colon ObI\rightarrow \mathbb{N}$ defined as the dimension of the over categories
\[\deg i=\dim NI/_i\]
It is easy to see that for any non-identity map $i\rightarrow j$ it satisfies $\deg(i)<\deg(j)$. This is sometimes called a directed Reedy category. For every positive integer $d$, define $I_{\leq d}$ to be the full subcategory of $I$ on objects of degree less than or equal to $d$, and $I_d$ the full subcategory of objects of degree $d$.

For $i$ of degree $-1$, the category $I_{\leq -1}$ is empty and $\widetilde{f}$ is the empty map.
Now suppose that $\widetilde{f}$ is defined as a natural transformation from the category $I_{\leq d-1}$, and let $i$ be an object of $I_d$.
By degree reasons, the only non-identity morphisms of $I_{\leq d}$ involving $i$ are maps $j\rightarrow i$ with $j$ in $I_{\leq d-1}$. In order to be compatible with $I_{\leq d-1}$ and to extend $\widetilde{g}$, the map $\widetilde{f}_{i}$ needs to satisfy the following extension problem in $Top_\ast$
\[\xymatrix{ K_{i}\wedge S^{k}\ar[r]\ar[dr]_{\widetilde{g}_{i}}&K_{i}\wedge D^{k+1}\ar@{-->}[d]^-{\widetilde{f}_{i}}&L_{i}(K)\wedge D^{k+1}\ar[l]\ar[d]^-{\widetilde{f}|_{I_{\leq d-1}}}\\
&X_{i}&L_{i}(X)\ar[l]
}\]
Here $L_i(Z)$ is the $i$-latching space of a diagram $Z\in Top_{\ast}^I$, with verticies $L_i(Z)=\colim\limits_{j\stackrel{\neq\id}{\rightarrow} i}Z_j$.
The right-hand square expresses that $\widetilde{f}_{i}$ needs to be compatible with the extensions previously defined on $I_{\leq d-1}$. Both horizontal maps in the first row are cofibrations since $K$ is cofibrant.
The extension problem above is equivalent to the extension problem
\[\xymatrix{ L_{i}(K)\wedge D^{k+1}\coprod\limits_{L_{i}(K)\wedge S^k}K_{i}\wedge S^k\ar[d]\ar[r]^-{}&X_{i}\\
K_{i}\wedge D^{k+1}\ar@{-->}[ur]_-{\widetilde{f}_{i}}
}\]
and the vertical map is also a cofibration. The extension $\widetilde{f}_{i}$ can be defined inductively on the relative cells of the cofibration, provided that for any $(n+1)$-cell the composition of $\widetilde{g}_{i}$ with the attaching map $S^{n}\rightarrow X_i$ is null-homotopic. If $K_i\wedge D^{k+1}$ has a $(n+1)$-cell, by dimension reasons we must have
\[n+1\leq \dim K_i\wedge D^{k+1}=\dim K_i+k+1\leq \conn X_{i}+1\]
The last inequality holds as $k$ is smaller than the range of the statement. Thus $\pi_n X_i$ is trivial, that is any map $S^{n}\rightarrow X_i$ is null-homotopic.
\end{proof}

\section{A refined cofinality theorem for homotopy limits}\label{cofinality}

Let $X\colon J\rightarrow\Top$ be a diagram of spaces, and $F\colon I\rightarrow J$ a functor. We aim at describing the homotopy fiber of the restriction map 
\[F^\ast\colon \holim_JX\longrightarrow \holim_IF^\ast X\]

\begin{defn}[\cite{thomason}]\label{grot}
The Grothendieck construction of a functor $\Phi\colon K\rightarrow Cat$ is the category $K\wr \Phi$ with objects $\coprod_{\Ob K}\Ob \Phi(k)$
and morphisms
\[\hom_{K\wr \Phi}((k,x),(l,y))=\coprod\limits_{\gamma\in\hom_K(l,k)}\hom_{\Phi(k)}\big(x,\Phi(\gamma)(y)\big)\]
Notice that the orientation of the arrows is reversed from the original definition of \cite{thomason}.
\end{defn}
A functor $F\colon I\rightarrow J$ defines a diagram of categories $\ast\leftarrow I\stackrel{F}{\rightarrow} J$ indexed on the punctured square $\mathcal{P}_0(1_+)=(\bullet\leftarrow\bullet\rightarrow \bullet)$. Its Grothendieck construction is called the cofiber of $F$ and it is denoted
\[\hoc F=\mathcal{P}_0(1_+)\wr(\ast\leftarrow I\stackrel{F}{\rightarrow} J)\]
The category $J$ includes as a full subcategory in $\hoc F$, and there is a sequence of functors
\[I\stackrel{F}{\longrightarrow} J\stackrel{\iota}{\longrightarrow} \hoc F\]
By Thomason's theorem \cite[1.2]{thomason} this is a categorical model for the homotopy cofiber of the nerve of $F$.

\begin{example}
Let $F\colon \mathcal{P}_0(1_+)\rightarrow \mathcal{P}_0(2_+)$ be the inclusion that sends $+$ to $+$, and that adds the element $2$ to the other subsets. The category $\hoc F$ is the poset
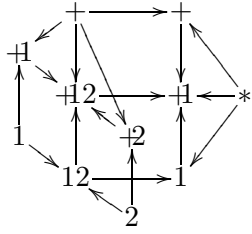
\[\xymatrix@1@=6pt{&+\ar[dd]\ar[dddr]\ar[dl]\ar[rr]&&+\ar[dd]\\
+\!\!1\ar[dr]\\
&+\!\!12\ar[rr]&&+\!\!1&&\ast\ar[uull]\ar[ll]\ar[ddll]\\
1\ar[uu]\ar[dr]&&+\!\!2\ar[ul]\\
&12\ar[uu]\ar[rr]&&1\ar[uu]\\
&&2\ar[uu]\ar[ul]
}\]
\end{example}

Let $X\colon J\rightarrow \Top$ be a diagram and let $\ast\rightarrow F^\ast X$ be a natural transformation. Notice that this is the same as the data of a point in the limit of $C$. This induces an extension $\overline{X}\colon\hoc F\rightarrow \Top$ of $X$ to the cofiber of $F$, defined by sending the objects $j\in J$
 to $X_j$, the objects $i\in I$ to $X_{F(i)}$ and the object $\ast$ to the one point space. On morphisms, it sends the maps $\ast\rightarrow i$ to the base point $\ast\rightarrow X_{F(i)}$, and the maps in $\hom_{\hoc F}(j,i)=\hom_J(j,F(i))$ to $X_j\rightarrow X_{F(i)}$. Altogether, we have a commutative diagram of functors
\[\xymatrix{I\ar[r]^F\ar[dr]_{F^\ast X}&J\ar[r]^-{\iota}\ar[d]^X&\hoc F\ar[dl]^{\overline{X}}\\
& \Top
}\] 

\begin{rem}
Given maps of spaces $Y\stackrel{f}{\rightarrow}Z\stackrel{g}{\rightarrow}T$, an extension of $g$ to the homotopy cofiber of $f$ is precisely a null-homotopy of $g\circ f$. The analogue of this null-homotopy for the functor case is a natural transformation $\ast\rightarrow F^\ast X$.
\end{rem}

We thank an anonymous referee for pointing out that the following theorem is a direct consequence of the Fubini's theorem for homotopy limit of \cite[31.5]{voj}. Theorem B below is the special case of a pushout diagram.

\begin{theoremB}[{cf.\cite[31.5]{voj}}]
Let $F\colon I\rightarrow J$ be a functor, $X\colon J\rightarrow\Top$ a diagram of spaces and $\ast\rightarrow F^{\ast}X$ a natural transformation. There is a homotopy cartesian square of spaces
\[\xymatrix{\displaystyle\holim_{\hoc F}\overline{X}\ar[r]^{\iota^\ast}\ar[d]&\displaystyle\holim_JX\ar[d]^{F^{\ast}}\\
\ast\ar[r]&\displaystyle\holim_IF^{\ast}X
}\]
where the bottom horizontal map is the constant natural transformation $BI/_{-}\rightarrow\ast\rightarrow F^{\ast}X$. 
\end{theoremB}

\begin{proof}
Since homotopy limits commute with limits, the homotopy fiber of $F^{\ast}$ over the constant natural transformation is homeomorphic to the equalizer
\[\lim\Big(\hofib\left(
\vcenter{\hbox{\xymatrix{\displaystyle\!\!\prod_{j\in \Ob J}\map(BJ/_j,X_j)\ar[d]\\
\displaystyle\!\!\prod_{i\in \Ob I}\map(BI/_i,X_{F(i)})}}}
\!\!\right)
\rightrightarrows \hofib\left(
\vcenter{\hbox{\xymatrix{\displaystyle\!\!\!\prod_{(\beta\colon j'\rightarrow j)\in \hom J}\map(BJ/_{j'},X_j)\ar[d]\\
\displaystyle\!\!\!\prod_{(\alpha\colon i'\rightarrow i)\in \hom I}\map(BI/_{i'}, X_{F(i)})}}}
\!\!\right)\Big)\]
where the homotopy fibers are taken over the base points of $X_{F(i)}$.
The left hand homotopy fiber is homeomorphic to
\[\prod_{i\in I}\hofib\Big(\map(BJ/_{F(i)},X_{F(i)})\stackrel{F^\ast}{\rightarrow} \map(BI/_i,X_{F(i)})\Big)\times \prod_{j\in Ob J\backslash F(ObI)}\!\!\!\!\!\!\map(BJ/_j,X_{j})\]
which is homeomorphic to
\[\prod_{i\in I}\map_\ast\big(\hoc(BI/_i\rightarrow BJ/_{F(i)}),X_{F(i)}\big)\times \prod_{j\in Ob J\backslash F(ObI)}\!\!\!\!\!\!\map(BJ/_j,X_{j})\]
where $\hoc$ denotes the cofiber and $\map_\ast$ is the space of pointed maps. Similarly the right hand homotopy fiber is homeomorphic to
\[\prod_{(\alpha\colon i'\rightarrow i)\in\hom_I}\!\!\!\! \map_\ast\big(\hoc(BI/_i\rightarrow BJ/_{F(i)}),X_{F(i)}\big)\times \!\!\!\!\!\prod_{(\beta\colon j'\rightarrow j)\in \hom_J\backslash F(\hom_I)}\!\!\!\!\!\!\!\!\!\!\!\!\!\!\map(BJ/_j,X_{j})\]
Let us denote $C_i=\hoc(BI/_i\rightarrow BJ/_{F(i)})$. By inspection, the equalizer above is the space of collections of pointed maps $\{f_i\colon C_i\rightarrow X_{F(i)}\}_{i\in Ob I}$ and maps $\{g_i\colon BJ/_j\rightarrow X_j\}_{i\in Ob J}$ subject to the commutativity of the following diagrams
\begin{multicols}{2}
\begin{enumerate}
\item For all $\alpha\colon i'\rightarrow i$
\[\xymatrix{C_{i'}\ar[r]^{f_{i'}}\ar[d]_{\alpha_\ast} &X_{F(i')}\ar[d]^{X(F(\alpha))}\\
C_{i}\ar[r]^{f_{i}}& X_{F(i)}
}\]
\phantom{a}
\item For all $\beta\colon j'\rightarrow j$ with $j,j'\notin F(\Ob I)$
\[\xymatrix{BJ/_{j'}\ar[r]^{g_{j'}}\ar[d]_{\beta_\ast} &X_{j'}\ar[d]^{X(\beta)}\\
BJ/_{j}\ar[r]^{g_{j}}& X_{j}
}\]
\item For all $\beta\colon j'\rightarrow j$ with $j'\!\notin\! F(\Ob I)$ and for all $i\in F^{-1}(j)$
\[\xymatrix{BJ/_{j'}\ar[rr]^-{g_{j'}}\ar[d]_{\beta_\ast}&&X_{j'}\ar[d]^{X(\beta)}\\
BJ/_{j}\ar[r]& C_i\ar[r]_-{f_i}& X_{j}
}\]
\item For all $\beta\colon j'\rightarrow j$ with $j\notin F(\Ob I)$ and for all $i'\in F^{-1}(j')$
\[\xymatrix{BJ/_{j'}\ar[d]_{\beta_\ast}\ar[r]&C_{i'}\ar[r]^-{f_{i'}}&X_{j'}\ar[d]^{X(\beta)}\\
BJ/_{j}\ar[rr]_-{g_j}&& X_{j}
}\]
\end{enumerate}
\end{multicols}
5. $\hbox{For all $\beta\colon j'\rightarrow j$ and for all $i\in F^{-1}(j)$ and $i'\in F^{-1}(j')$}$
\[\xymatrix{BJ/_{j'}\ar[d]_{\beta_\ast}\ar[r]&C_{i'}\ar[r]^-{f_{i'}}&X_{j'}\ar[d]^{X(\beta)}\\
BJ/_{j}\ar[r]&C_i\ar[r]_-{f_i}& X_{j}
}\]

This is exactly the space of natural transformations $\hom(K,\overline{X})$, for the diagram of spaces $K\colon \hoc F\rightarrow \Top$ defined at objects $j$ of $J$ by $K_j=BJ/_j$, at objects $i$ of $I$ by $K_i=C_i$ and at the point $\ast$ by the one point space. A morphisms $\gamma$ of $\hom_{\hoc F}(j,i)=\hom_J(j, F(i))$ is sent to the composite 
\[BJ/j\stackrel{\gamma_\ast}{\longrightarrow} BJ/_{F(i)}\longrightarrow C_i\] 
where the second map is the canonical inclusion of the cone. For the objects of $\hoc F$ coming from $J$ and for the base point, it is clear that $(\hoc F)/_j=J/_j$ and $(\hoc F)/_\ast=\ast$. Moreover at the objects of $i$, there are natural maps
\[C_i\cong\hoc(BI/_i\stackrel{BF}{\rightarrow}BJ/_{F(i)})\stackrel{\simeq}{\longrightarrow} B\hoc(I/_i\stackrel{F}{\rightarrow}J/_{F(i)})\cong B(\hoc F)/_i
\]
where the second map is the natural equivalence of Thomason's theorem \cite[1.2]{thomason}.
This gives an equivalence of diagrams of spaces $K\stackrel{\simeq}{\rightarrow}B(\hoc F)/_{-}$. Both diagrams are cofibrant in the projective model structure, the target because it is a diagram of over categories (see \cite[14.8.9]{Hirsch}) and the source by \ref{cofibrancy} below. As $\hom(-,\overline{X})$ preserves equivalences of cofibrant diagrams, we get an equivalence
\[\holim_{\hoc F}\overline{X}=\hom(B(\hoc F)/_{-},\overline{X})\stackrel{\simeq}{\longrightarrow}\hom(K,\overline{X})\]
with the target homeomorphic to the homotopy fiber of the restriction map.
\end{proof}

\begin{lemma}\label{cofibrancy}
The diagram $K\colon \hoc F\rightarrow \Top$ from the proof of Theorem B is cofibrant in the projective model structure of the diagram category $\Top^{\hoc F}$.
\end{lemma}

\begin{proof}
We need to solve the lifting problem in the category $\Top^{\hoc F}$
\[\xymatrix{&E\ar@{->>}[d]^{\sim}\\
K\ar[r]\ar@{-->}[ur]^-{l}&B
}\]
where the vertical map is an acyclic fibration. The restriction of $K$ to $J$ is by definition $K|_J=NJ/_{-}$, which is cofibrant in $\Top^J$ (see \cite[14.8.9]{Hirsch}). Therefore we can choose lifts $l|_J$ and $l_{\ast}$ respectively on the subcategories $J$ and $\ast$ of $\hoc F$
\[\xymatrix{&E|_J\ar@{->>}[d]^{\sim}\\
K|_J\ar[r]\ar[ur]^-{l_J}&B|_J
} \ \ \ \ \ \ \ \ \ \ \ \ \ \ \ \xymatrix{&E_\ast\ar@{->>}[d]^{\sim}\\
K_\ast=\ast\ar[r]\ar[ur]^-{\ast}&B_\ast
}\]
where the right hand diagram is a diagram of simplicial sets. We are left with defining the lift $l$ on $I$, compatible with the maps $j\rightarrow i$ and $\ast\rightarrow i$ in $\hoc F$. Notice that every map $j\rightarrow i$ factors through the canonical map $F(i)\rightarrow i$ in $\hoc F$ defined by the identity on $F(i)$. Therefore a compatible lift $l_I$ is given by the lifting problem in $\Top^I$
\[\xymatrix{K_{F(-)}\coprod\ast\ar[d]\ar[rr]^-{(l_J)|_{F(-)}\coprod l_\ast}&&E|_{F(-)}\coprod E_\ast\ar[r]&E|_I\ar@{->>}[d]^{\sim}\\
K|_I\ar[rrr]\ar@{-->}[urrr]_-{l_I}&&&B|_I
}\]
where the top triangle expresses compatibility with $l_J$ and $l_\ast$. This lifting problem has a solution as long as the left vertical map is a cofibration in $\Top^I$. This is the right vertical map in the diagram
\[\xymatrix{NI/_{-}\coprod NI/_{-}\ar[r]\ar[d]&NI/_{-}\coprod \ast\ar[r]\ar[d]&NJ/_{F(-)}\coprod\ast=K|_{F(-)}\coprod\ast\ar@<4ex>[d]\\
NI/_{-}\times [0,1]\ar[r]&CNI/_{-}\ar[r]&C_{(-)}=K|_I
}\]
The two small squares are pushouts, and the left vertical map is the inclusion of the top and the bottom of the cylinder, which is a cofibration. Therefore so is $K|_{F(-)}\coprod\ast\rightarrow K|_I$.
\end{proof}

\bibliographystyle{amsalpha}
\bibliography{holims}

\end{document}